\documentclass[reqno]{amsart}
\usepackage[english]{babel}
\usepackage[T1]{fontenc}
\usepackage[utf8]{inputenc}
\usepackage{amsmath}
\usepackage{amsfonts}
\usepackage{array}
\usepackage{eurosym}
\usepackage[noend,linesnumbered]{algorithm2e}
\usepackage{amssymb}
\usepackage{mathrsfs}
\usepackage{amsthm}
\usepackage{xfrac}
\usepackage{lmodern}
\usepackage{fix-cm}
\usepackage{listings}
\usepackage{fancyvrb}
\usepackage{enumerate}
\usepackage{xcolor}
\usepackage{tikz}
\tikzstyle{every picture}=[line width=.7pt,minimum size=3pt,every label/.append style={font=\normalsize},label distance=2pt]
\tikzstyle{every node}=[font=\normalsize,circle,draw=black,fill=black,inner sep=0pt,minimum width=1.3pt]
\usepackage{mathrsfs}
\usetikzlibrary{arrows}
\usepackage{float}
\usepackage{geometry}
\usepackage{caption}
\usepackage{subcaption}
\usepackage{color}
\usepackage[bookmarks=true]{hyperref}
\usepackage{wasysym}
\usepackage{enumitem}
\usepackage{longtable}

\makeatletter
\newcolumntype{"}{@{\hskip\tabcolsep\vrule width 1pt\hskip\tabcolsep}}
\makeatother

\theoremstyle{plain}
\newtheorem{theorem}{Theorem}[section]
\newtheorem{lemma}[theorem]{Lemma}
\newtheorem{proposition}[theorem]{Proposition}

\newtheorem{myalgorithm}[theorem]{Algorithm}

\newtheorem{conjecture}[theorem]{Conjecture}
\newtheorem{question}[theorem]{Question}

\theoremstyle{definition}
\newtheorem{definition}[theorem]{Definition}
\newtheorem{example}[theorem]{Example}
\newtheorem{notation}[theorem]{Notation}

\theoremstyle{remark}
\newtheorem{remark}[theorem]{Remark}

\newcommand{\CC}{\mathcal{C}}

\newcommand{\hgt}{{\rm ht}}

\newcommand{\mc}{\mathcal}

\newcommand{\overbar}[1]{\mkern 1.5mu\overline{\mkern-3mu#1\mkern-1.5mu}\mkern 1.5mu}
\newcommand{\B}{\overbar{B}}
\newcommand{\C}{\overbar{C}}

\title{Cohen-Macaulay binomial edge ideals of small graphs}

\author[D. Bolognini, A. Macchia, G. Rinaldo, F. Strazzanti]{Davide Bolognini, Antonio Macchia, Giancarlo Rinaldo, Francesco Strazzanti}\thanks{Antonio Macchia was supported by the Deutsche Forschungsgemeinschaft (DFG, German Research Foundation) – project number 454595616.}

\address{{\small Davide Bolognini, Dipartimento di Ingegneria Industriale e Scienze Matematiche, Universit\`a Politecnica delle Marche, Via Brecce Bianche, 60131 Ancona, Italy}}
\email{{\small davide.bolognini.cast@gmail.com}}

\address{{\small Antonio Macchia, Fachbereich Mathematik und Informatik, Freie Universit\"at Berlin, Arnimallee 2, 14195 Berlin, Germany}}
\email{{\small macchia.antonello@gmail.com}}

\address{{\small Giancarlo Rinaldo, Dipartimento di Matematica e Informatica, Fisica e Scienze della Terra, Università di Messina, Viale Ferdinando Stagno d’Alcontres 31, 98166 Messina, Italy}}
\email{{\small giancarlo.rinaldo@unime.it}}

\address{{\small Francesco Strazzanti, Dipartimento di Matematica ``Giuseppe Peano'', Universit\`a degli Studi di Torino, Via Carlo Alberto 10, 10123 Torino, Italy}}
\email{{\small francesco.strazzanti@gmail.com}}



\begin{document}

\begin{abstract}
A combinatorial property that characterizes Cohen-Macaulay binomial edge ideals has long been elusive. A recent conjecture ties the Cohen-Macaulayness of a binomial edge ideal $J_G$ to special disconnecting sets of vertices of its underlying graph $G$, called \textit{cut sets}. More precisely, the conjecture states that $J_G$ is Cohen-Macaulay if and only if $J_G$ is unmixed and the collection of the cut sets of $G$ is an accessible set system. In this paper we prove the conjecture theoretically for all graphs with up to $12$ vertices and develop an algorithm that allows to computationally check the conjecture for all graphs with up to $15$ vertices and all blocks with whiskers where the block has at most $11$ vertices. This significantly extends previous computational results.
\end{abstract}

\maketitle

\section{Introduction}

Finding a combinatorial characterization of an algebraic property of a class of ideals in a polynomial ring is an interesting and challenging task. A prominent example is Fr\"oberg's theorem according to which the (monomial) edge ideal of a graph $G$ has a linear resolution if and only if the complement of $G$ is chordal \cite{F90}.

In this paper we study a graph-theoretical characterization of Cohen-Macaulay binomial edge ideals, a class of binomial ideals naturally defined starting from simple graphs. Fix a field $K$. Given a finite simple graph $G$ with vertex set $V(G)$ and edge set $E(G)$, the \textit{binomial edge ideal} of $G$ is
\[
J_G = (x_i y_j - x_j y_i \mid \{i,j\} \in E(G)) \subset R = K[x_1, \dots, x_m, y_1, \dots, y_m],
\]
where $m = |V(G)|$. Binomial edge ideals have been introduced independently in \cite{HHHKR10} and \cite{O11} and generalize the well-studied ideals of $2$-minors of a generic $2 \times m$ matrix.

Crucial in the theory of binomial edge ideals is the notion of cut sets, special subsets of vertices that disconnect the graph. More precisely, given a subset $S \subset V(G)$, we denote by $G \setminus S$ the subgraph induced by $G$ on the vertices $V(G) \setminus S$ and by $c_G(S)$ the number of connected components of $G \setminus S$. A subset $S \subset V(G)$ is called a \textit{cut-point set}, or simply \textit{cut set}, of $G$ if either $S = \emptyset$ or $c_G(S) > c_G(S \setminus \{i\})$ for every $i \in S$. We call \textit{cut vertex} or \textit{cut-point} a vertex $v$ of $G$ such that $\{v\}$ is a cut set of $G$. We denote by $\mc C(G)$ the collection of cut sets of $G$. Cut sets are important because there is a one-to-one correspondence between the cut sets of $G$ and the minimal primes of $J_G$, see \cite[Corollary 3.9]{HHHKR10}.

The study of Cohen-Macaulay binomial edge ideals mainly focused on the search of classes of graphs \cite{BMS18, EHH11, R13, R19} and of constructions preserving this property \cite{KS15, RR14}. On the other hand, in \cite{AM20} it is described a topological characterization in terms of reduced homology groups of a certain poset, but there are no general combinatorial characterizations so far. In \cite{BMS22} the first, second and fourth authors propose the following conjecture in terms of the structure of the cut sets of the associated graph.

\begin{conjecture}[{\cite[Conjecture 1.1]{BMS22}}] \label{conjecture}
Let $G$ be a graph. Then $R/J_G$ is Cohen-Macaulay if and only if $G$ is accessible.
\end{conjecture}

Following \cite{BMS22}, recall that a graph $G$ is called \textit{accessible} if $J_G$ is unmixed and $\mc C(G)$ is an \textit{accessible set system} (see \cite[p. 360]{KL83}), i.e., for every non-empty $S \in \mc C(G)$ there exists $s \in S$ such that $S \setminus \{s\} \in \mc C(G)$. This notion first emerged in \cite{BMS18}, where Conjecture \ref{conjecture} was proved for bipartite graphs.
Accessibility is a completely combinatorial notion since the unmixedness of $J_G$ can also be expressed in terms of the cut sets of $G$, see \cite[Lemma 2.5]{RR14}: $J_G$ is unmixed if and only if $c_G(S) = |S| + c$ for every $S \in \CC(G)$, where $c$ is the number of connected components of $G$.

Conjecture \ref{conjecture} holds for chordal, bipartite and traceable graphs, see \cite[Theorems 6.4 and 6.8, Corollary 6.9]{BMS22}, as well as for other classes of graphs described in \cite{LMRR21} and \cite{SS22}.
The proof of these results relies on a further condition, the \textit{strong unmixedness} of $J_G$, introduced in \cite{BMS22}, that only depends on $G$ (and not on the field $K$). For every graph $G$, the following implications hold as a consequence of \cite[Theorem 5.11]{BMS22} and \cite[Theorem 2]{LMRR21}:
\begin{equation} \label{Conditions}
J_G \text{ strongly unmixed} \Rightarrow R/{J_G} \text{ Cohen-Macaulay} \Rightarrow R/{J_G} \text{ satisfies Serre's condition } (S_2) \Rightarrow G \text{ accessible}.
\end{equation}

In this article we provide both theoretical and computational evidence for Conjecture \ref{conjecture}, proving that it holds for new classes of graphs. To this aim, we develop a series of theoretical tools that allows us to limit the search space to a special family of graphs.

Recall that a \textit{block}, or biconnected graph, is a graph that does not have cut vertices and \textit{adding a whisker} to a vertex $v$ of a graph means attaching a pendant edge $\{v,f\}$, where $f$ is a new vertex. In particular, $f$ is a \textit{free vertex}, which means that it belongs to a unique maximal clique. The first important reduction to prove that $G$ accessible implies $J_G$ strongly unmixed is to consider only \textit{block with whiskers}, i.e., graphs constructed by adding a whisker to some vertices of a block. This follows from \cite[Proposition 3]{LMRR21} and \cite[Theorem 1.2]{SS22}, see also \cite[Corollary 3.13]{SS22b}. By focusing on blocks with whiskers we prove:

\begin{theorem} \label{T.new_cases_conj}
Let $G$ be one of the following:
\begin{itemize}
\item[{\rm (a)}] a block with $n$ vertices and $k \geq n-2$ whiskers;
\item[{\rm (b)}] a block with whiskers, where the vertices of the block are at most $11$;
\item[{\rm (c)}] a graph with up to $15$ vertices.
\end{itemize}
Then the conditions in \eqref{Conditions} are all equivalent. In particular, Conjecture \ref{conjecture} holds for all the graphs above and in these cases the Cohen-Macaulayness of $R/J_G$ does not depend on the field.
\end{theorem}

More precisely,
\begin{itemize}
\item we prove case (a) in Theorem \ref{T.Conjecture_n=k+2} and case (c) for graphs up to $12$ vertices only using theoretical results in Theorem \ref{T.12_vertices}. The latter recovers the computational results of \cite{LMRR21}, in which an exhaustive search on all connected graphs took about a month on a high-performance computing system;
\item we develop a procedure, Algorithm \ref{algorithm}, to computationally treat cases (b) and (c), see Theorem \ref{T.computational_results}. A complete implementation of Algorithm \ref{algorithm} using \textit{C}, \textit{C++}, \textit{Python}, and the packages \textit{Nauty} \cite{Nauty} and \textit{igraph} \cite{igraph} can be found in the website \cite{BMRS23}, where we also provide some examples of computation.
\end{itemize}

The key part in the proof of Theorem \ref{T.new_cases_conj} consists in answering Questions \ref{Q.Main_question}, which asks whether for a given graph $G$ there exists a cut vertex for which $J_{G \setminus \{v\}}$ is unmixed. In Theorem \ref{T.minus_v_unmixed} we prove that a positive answer to this question for blocks with whiskers implies the equivalence of the conditions in \eqref{Conditions} for all graphs. In particular, we find this cut vertex for the following classes of blocks with whiskers:

\begin{theorem}\label{T.filters}
Let $B$ be a block with $n$ vertices and $\B$ be the graph obtained by adding $k>0$ whiskers to $B$.
Assume that $\B$ is accessible and satisfies one of the following properties:
\begin{enumerate}
  \item $B$ contains a free vertex (Proposition \ref{P.blockWithFreeVertex});
  \item $B$ has a vertex of degree at most two (Proposition \ref{P.blocksWithDegree2verticesAreGood});
  \item $\B$ has $k \leq 3$ whiskers (Proposition \ref{P.three_whiskers});
  \item there is a cut vertex $v$ of $\B$ such that $|N_B(v)| \geq \lfloor \frac{n+r}{2} \rfloor - 1$, where $r$ is the number of cut vertices adjacent to $v$ plus one (Proposition \ref{P.BoundNeighbourhood});
  \item $\B$ has $k=4$ whiskers and the induced subgraph on the cut vertices of $\B$ is a block (Proposition \ref{P.FourWhiskers});
  \item $\B$ has $k \geq n-2$ whiskers (Proposition \ref{P.blocksWithTwoNonCutvertex}).
\end{enumerate}
Then there exists a cut vertex of $\B$ for which $J_{\B \setminus \{v\}}$ is unmixed.
\end{theorem}

We remark here that, filtering the blocks with whiskers by using Theorem \ref{T.filters}, now Algorithm \ref{algorithm} allows to test the equivalence of the conditions in \eqref{Conditions} for all graphs up to $12$ vertices in a few seconds.

In Proposition \ref{P.blocksWithOneNonCutvertex} we also prove that, for blocks with $k \geq n - 1$ whiskers all the conditions in \eqref{Conditions} are equivalent to the unmixedness of $J_G$, whereas this is not the case for $k = n - 2$, see Example \ref{E.k=n-2}.

Part of the results and computations contained in this paper were announced in the extended abstract \cite{BMRS22}.

\section{Strongly unmixed binomial edge ideals}

Let $G$ be a finite simple graph, i.e., a graph without loops and multiple edges, with vertex set $V(G)$ and edge set $E(G)$.
Given $v \in V(G)$, we denote by $G_v$ the graph with vertex set $V(G_v)=V(G)$ and edge set $E(G_v)=E(G) \cup \{\{w_1,w_2\} \mid \{v,w_1\},\{v,w_2\} \in E(G)\}$. We recall the following definition from \cite[Definition 5.6]{BMS22}:

\begin{definition}
A binomial edge ideal $J_G$ is called \textit{strongly unmixed} if the connected components of $G$ are complete graphs or if $J_G$ is unmixed and there exists a cut vertex $v$ of $G$ such that $J_{G \setminus \{v\}}$, $J_{G_v}$, and $J_{G_v \setminus \{v\}}$ are strongly unmixed.
\end{definition}

We notice here that in the above definition it is enough to require that only one between $J_{G_v}$ and $J_{G_v \setminus \{v\}}$ is strongly unmixed.

\begin{proposition} \label{P.stronglyUnmixed}
Let $J_G$ be unmixed. The following conditions are equivalent:
\begin{enumerate}
\item $J_G$ is strongly unmixed;
\item There exists a cut vertex $v$ of $G$ such that $J_{G \setminus \{v\}}$ and $J_{G_v}$ are strongly unmixed;
\item There exists a cut vertex $v$ of $G$ such that $J_{G \setminus \{v\}}$ and $J_{G_v \setminus \{v\}}$ are strongly unmixed.
\end{enumerate}
\end{proposition}

\begin{proof}
Clearly (1) implies (2) and (3). Since $J_G$ and $J_{G\setminus \{v\}}$ are unmixed, \cite[Lemma 3.16]{SS22} says that if (2) holds, then $J_{G_v \setminus \{v\}}$ is strongly unmixed. Assume now that (3) holds. Since $J_{G}$ is unmixed, also $J_{G_v}$ is unmixed by \cite[Lemma 4.5]{BMS22}. Moreover, $v$ is a free vertex of $G_v$ and then \cite[Lemma 11]{LMRR21} implies that $J_{G_v}$ is strongly unmixed.
\end{proof}

Note that, as for accessibility, the strong unmixedness of $J_G$ is also a purely combinatorial property, i.e., it depends only on $G$ and not on the field $K$. By \cite[Theorem 5.11]{BMS22} and \cite[Theorem 2]{LMRR21} for any graph $G$ we have:
\begin{equation}
J_G \text{ strongly unmixed} \Rightarrow R/{J_G} \text{ Cohen-Macaulay} \Rightarrow R/{J_G} \text{ satisfies Serre's condition } (S_2) \Rightarrow G \text{ accessible.} \tag{\ref{Conditions}}
\end{equation}
It is not known whether any of these implications can be reversed. In order to show that $G$ accessible implies $J_G$ strongly unmixed, i.e., the above conditions are all equivalent, as a consequence of \cite[Proposition 5.13 and Corollary 5.16]{BMS22} it is enough to answer in the affirmative the following question:

\begin{question}[{\cite[Question 5.17]{BMS22}}] \label{Q.Main_question}
If $G$ is a connected non-complete accessible graph, does there exist a cut vertex $v$ of $G$ such that $J_{G\setminus \{v\}}$ is unmixed?
\end{question}

If $H$ is an induced subgraph of $G$, we denote by $N_H(v)=\{w \in V(H) \mid \{v,w\} \in E(G)\}$ the set of neighbors of $v$ in $H$ and define $N_H[v] = N_H(v) \cup \{v\}$. In order to study Question \ref{Q.Main_question} we frequently use the following criterion:

\begin{proposition}[{\cite[Proposition 5.2]{BMS22}}]\label{P.good_cut_vertex}
Let $v$ be a cut vertex of a connected graph $G$ and assume that $J_G$ is unmixed. Let $H_1$ and $H_2$ denote the connected components of $G\setminus \{v\}$. The
following statements are equivalent:
\begin{enumerate}
\item $J_{G \setminus \{v\}}$ is unmixed;
\item if $S$ is a cut set of $G \setminus \{v\}$, then $N_{H_1}(v) \not\subseteq S$ and $N_{H_2}(v) \not\subseteq S$.
\end{enumerate}
\end{proposition}

Note that, in an accessible graph $G$ it is not necessarily true that $J_{G \setminus \{v\}}$ is unmixed for every cut vertex $v$ of $G$.

\begin{example}
Let $G$ be the accessible graph in Figure \ref{F.oneGoodCutVertex}, whose cut vertices are $1,4,8$. In this case, $J_{G \setminus \{1\}}$ is unmixed, whereas $J_{G \setminus \{4\}}$ and $J_{G \setminus \{8\}}$ are not unmixed by Proposition \ref{P.good_cut_vertex}. In fact,
\begin{itemize}
\item for the cut vertex $4$, if $H_1 = G \setminus \{4,12\}$, then $N_{H_1}(4) \subset S_1 = \{1,3,5,8\}$, where $S_1$ is a cut sets of $G \setminus \{4\}$;
\item for the cut vertex $8$, if $H_1 = G \setminus \{8,13\}$, then $N_{H_1}(8) \subset S_2 = \{1,4,7,9\}$, where $S_2$ is a cut sets of $G \setminus \{8\}$.
\end{itemize}

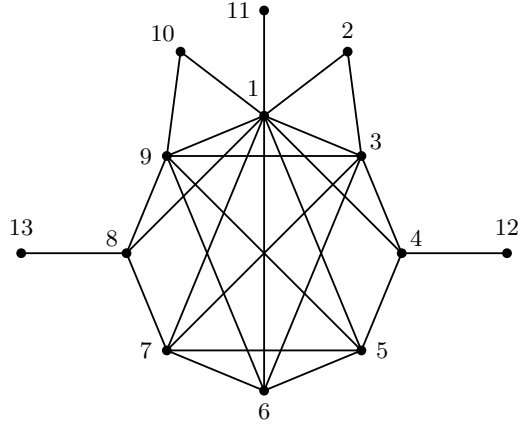
\begin{figure}[ht!]
\begin{tikzpicture}[scale=0.7]
\node[label={below:{\small $6$}}] (a) at (1,-0.19891236737965817) {};
\node[label={left:{\small $7$}}] (a) at (-0.847759065022574,0.5664544973505217) {};
\node[label={right:{\small $5$}}] (a) at (2.847759065022574,0.566454497350521) {};
\node[label={above right:{\small $4$}}] (a) at (3.613125929752754,2.414213562373095) {};
\node[label={above right:{\small $3$}}] (a) at (2.847759065022575,4.261972627395669) {};
\node[label={[label distance=2mm]100:{\small $1$}}] (a) at (1,5.027339492125849) {};
\node[label={left:{\small $9$}}] (a) at (-0.847759065022573,4.26197262739567) {};
\node[label={above left:{\small $8$}}] (a) at (-1.6131259297527538,2.4142135623730967) {};
\node[label={above:{\small $13$}}] (a) at (-3.6131259297527536,2.414213562373095) {};
\node[label={above:{\small $12$}}] (a) at (5.613125929752754,2.414213562373095) {};
\node[label={left:{\small $11$}}] (a) at (1,7.027339492125849) {};
\node[label={above left:{\small $10$}}] (a) at (-0.5867066805824697,6.244862350143292) {};
\node[label={above:{\small $2$}}] (a) at (2.586706680582472,6.24486235014329) {};
\draw (-0.847759065022574,0.5664544973505217)-- (1,-0.19891236737965817);
\draw (1,-0.19891236737965817)-- (2.847759065022574,0.566454497350521);
\draw (2.847759065022574,0.566454497350521)-- (3.613125929752754,2.414213562373095);
\draw (3.613125929752754,2.414213562373095)-- (2.847759065022575,4.261972627395669);
\draw (2.847759065022575,4.261972627395669)-- (1,5.027339492125849);
\draw (1,5.027339492125849)-- (-0.847759065022573,4.26197262739567);
\draw (-0.847759065022573,4.26197262739567)-- (-1.6131259297527538,2.4142135623730967);
\draw (-1.6131259297527538,2.4142135623730967)-- (-0.847759065022574,0.5664544973505217);
\draw (1,7.027339492125849)-- (1,5.027339492125849);
\draw (3.613125929752754,2.414213562373095)-- (5.613125929752754,2.414213562373095);
\draw (-1.6131259297527538,2.4142135623730967)-- (-3.6131259297527536,2.414213562373095);
\draw (-0.5867066805824697,6.244862350143292)-- (-0.847759065022573,4.26197262739567);
\draw (-0.5867066805824697,6.244862350143292)-- (1,5.027339492125849);
\draw (1,5.027339492125849)-- (2.586706680582472,6.24486235014329);
\draw (2.586706680582472,6.24486235014329)-- (2.847759065022575,4.261972627395669);
\draw (1,5.027339492125849)-- (-1.6131259297527538,2.4142135623730967);
\draw (1,5.027339492125849)-- (-0.847759065022574,0.5664544973505217);
\draw (1,5.027339492125849)-- (1,-0.19891236737965817);
\draw (1,5.027339492125849)-- (2.847759065022574,0.566454497350521);
\draw (1,5.027339492125849)-- (3.613125929752754,2.414213562373095);
\draw (-0.847759065022573,4.26197262739567)-- (2.847759065022575,4.261972627395669);
\draw (-0.847759065022574,0.5664544973505217)-- (2.847759065022574,0.566454497350521);
\draw (-0.847759065022573,4.26197262739567)-- (2.847759065022574,0.566454497350521);
\draw (-0.847759065022573,4.26197262739567)-- (1,-0.19891236737965817);
\draw (2.847759065022575,4.261972627395669)-- (-0.847759065022574,0.5664544973505217);
\draw (2.847759065022575,4.261972627395669)-- (1,-0.19891236737965817);
\end{tikzpicture}
\caption{A graph $G$ with $J_{G \setminus \{1\}}$ unmixed} \label{F.oneGoodCutVertex}
\end{figure}
\end{example} 

\section{Blocks with whiskers}

In this section we show that to answer Question \ref{Q.Main_question} we can restrict to consider blocks with whiskers. Moreover, given a block with whiskers $G$, we find several sufficient conditions on the structure of $G$ for the existence of a cut vertex $v$ such that $J_{G \setminus \{v\}}$ is unmixed. Combining these results, we prove that $G$ accessible implies $J_G$ strongly unmixed, and hence Conjecture \ref{conjecture} holds true, for new classes of graphs and for all graphs with up to $12$ vertices.

Recall that, for a connected graph $G$, $J_G$ is unmixed if and only if $c_G(S) = |S| + 1$ for every $S \in \CC(G)$, see \cite[Lemma 2.5]{RR14}. For unmixed binomial edge ideals we characterize the cut sets of $G$ as the subsets $S$ of $V(G)$ satisfying a simple bound on the number of components of $G \setminus S$. This result will be useful throughout the paper.

\begin{lemma}\label{L.CutsetsInUnmixedGraphs}
Let $G$ be a connected graph with $J_G$ unmixed and $S$ be a subset of $V(G)$. The following are equivalent:
\begin{enumerate}
\item $S$ is a cut set of $G$;
\item $c_G(S) = |S|+1$;
\item $c_G(S) \geq |S|+1$.
\end{enumerate}
\end{lemma}

\begin{proof} Let $m = |V(G)|$. The implications $(1)\Rightarrow (2) \Rightarrow (3)$ are clear. Recall that, for every $S \subseteq V(G)$, there is a prime ideal $P_S(G)$ containing $J_G$ with $\hgt(P_S(G)) = m - c_G(S) + |S|$, see \cite[Lemma 3.1]{HHHKR10}. We refer the reader to \cite[Section 3]{HHHKR10} for more details. Thus, condition (3) implies that
\[
\hgt(P_S(G)) = m - c_{G}(S) + |S| \leq m - 1 = m - c_{G}(\emptyset) + |\emptyset| = \hgt(P_{\emptyset}(G)) = \hgt(J_G).
\]
Since $P_S(G)$ contains $J_G$, it follows that $\hgt(P_S(G)) = \hgt(J_G)$. Therefore $P_S(G)$ is a minimal prime of $J_G$ and \cite[Corollary 3.9]{HHHKR10} implies that $S$ is a cut set of $G$.
\end{proof}

\begin{remark} \label{R.union}
Let $G$ be a connected graph with $J_G$ unmixed. If $S$ is a cut set of $G$ and $v \in V(G) \setminus S$, Lemma \ref{L.CutsetsInUnmixedGraphs} implies that $S \cup \{v\}$ is a cut set if and only if $v$ is a cut vertex of $G \setminus S$.
\end{remark}

\subsection{Reduction to block with whiskers}
Recall that a \textit{block} is a graph with no cut vertices. A block of a graph $G$ is a maximal induced subgraph of $G$ that is a block. Any graph can be decomposed into blocks, where each two blocks share at most one cut vertex. 

Let $G$ be a connected graph and $B$ a block of $G$. Denote by $W=\{v_1,\dots,v_r\}$ the set of cut vertices of $G$ that are vertices of $B$. Then
\begin{equation}\label{eq:blockandGi}
G = B \cup \left(\bigcup_{i=1}^r G_i\right),
\end{equation}
where $V(G_i)\cap V(B)=\{v_i\}$ for $i=1,\ldots,r$, and the connected components of $G \setminus W$ are $G_1\setminus\{v_1\}, \dots, G_r\setminus\{v_r\}$ and $B\setminus W$ if it is non-empty.

If $B$ is a block and $W = \{v_1,\dots,v_r\} \subseteq V(B)$, we denote by $\B^W$ the graph obtained by adding a whisker to each vertex in $W$. More precisely:
\begin{enumerate}
    \item $V(\B^W)=V(B) \cup \{f_1,\ldots,f_r\}$, where $f_1,\dots,f_r$ are new vertices;
    \item $E(\B^W)=E(B) \cup \{\{v_i,f_i\} : i = 1,\dots,r\}$, where $\{v_i,f_i\}$ is the whisker attached to $v_i$.
\end{enumerate}
We call $\B^W$ a \textit{block with whiskers}. Whenever we do not need to specify the set $W$, we simply write $\B$.

\medskip

The next result implies that if Question \ref{Q.Main_question} has a positive answer for all blocks with whiskers, then Conjecture \ref{conjecture} holds for all graphs.

\begin{theorem}\label{T.minus_v_unmixed}
If in every accessible block with whiskers $G$ there is a cut vertex $v$ such that $J_{G \setminus \{v\}}$ is unmixed, then the binomial edge ideal of every accessible graph is strongly unmixed.
\end{theorem}

\begin{proof}
Assume by contradiction that there exist accessible graphs whose binomial edge ideal is not strongly unmixed and among them choose a graph $H$ with the minimum number of vertices. Clearly, $H$ is connected and not complete.

We claim that there are no cut vertices of $H$ such that $J_{H \setminus \{v\}}$ is unmixed. In fact if there is a cut vertex $v$ such that $J_{H \setminus \{v\}}$ is unmixed, \cite[Corollary 5.16]{BMS22} ensures that both $H \setminus \{v\}$ and $H_v \setminus \{v\}$ are accessible. Since they have fewer vertices than $H$, it follows that $J_{H \setminus \{v\}}$ and $J_{H_v \setminus \{v\}}$ are strongly unmixed. By Proposition \ref{P.stronglyUnmixed}, this implies that $J_H$ is strongly unmixed, a contradiction.

By \cite[Proposition 6.1]{BMS22}, we have that there is a block $B$ of $H$ such that for every cut vertex $v$ of $H$ contained in $V(B)$ there exists a cut set $S_v$ of $H \setminus \{v\}$ containing $N_B(v)$.

By \cite[Proposition 3]{LMRR21}, the graph $\B$ is accessible. By assumption there exists a cut vertex $v$ of $\B$ such that $J_{\B \setminus \{v\}}$ is unmixed. Clearly, the cut vertices of $\B$ are exactly the cut vertices of $H$ contained in $V(B)$. Hence, we already know that there exists a cut set $S_v$ of $H \setminus \{v\}$ containing $N_B(v)$. Let $T_v=S_v \cap V(B)$. Since $J_{\B \setminus \{v\}}$ is unmixed, $T_v$ cannot be a cut set of $\B\setminus \{v\}$, see Proposition \ref{P.good_cut_vertex}. Therefore, there exists $w \in T_v$ such that $c_{\B \setminus \{v\}}(T_v) \leq c_{\B\setminus \{v\}}(T_v \setminus \{w\})$. Notice that $w$ is a cut vertex of $\B$ because otherwise $S_v$ is a cut set of $H \setminus \{v\}$. Hence, by definition of $\B$, $w$ is adjacent to a leaf in $\B$ and then $c_{\B\setminus \{v\}}(T_v) \leq c_{\B\setminus \{v\}}(T_v \setminus \{w\})$ implies that $N_B(w) \setminus \{v\}$ is contained in $T_v$. Let $H_1$ and $H_2$ be the connected components of $H \setminus \{w\}$, where $H_2$ contains $B \setminus \{w\}$. Consider $U_1=S_v \cap H[V(H_1) \cup \{w\}]$, where $H[V(H_1) \cup \{w\}]$ denotes the subgraph induced by $H$ on the vertices $V(H_1) \cup \{w\}$. Since $S_v$ is a cut set of $H \setminus \{v\}$ and $N_{B}(w) \setminus \{v\} \subseteq T_v$, it is easy to see that $U_1$ is a cut set of $H \setminus \{v\}$ and of $H$.

We also know that there exists a cut set $S_w$ of $H\setminus \{w\}$ containing $N_B(w)$. It is straightforward to check that $U_2=S_w \cap V(H_2)$ is still a cut set of $H \setminus \{w\}$ containing $N_B(w)$; in particular, $U_2$ is also a cut set of $H$. Finally, consider $U_1 \cup U_2$. By construction of $U_1$ and $U_2$, it follows that $U_1 \cup U_2$ is a cut set of $H$. However, since $U_1 \subseteq V(H_1) \cup \{w\}$ and $U_2 \subseteq V(H_2)$ are cut sets of $H$, $w \in U_1$, and $N_B(w) \subseteq U_2$, we have that
\[
|U_1 \cup U_2| = c_H(U_1 \cup U_2)-1 = (c_H(U_1)-1) + (c_H(U_2)-1) -1 = |U_1| + |U_2| - 1 = |U_1 \cup U_2| - 1,
\]
where the first and third equalities follow by the unmixedness of $J_H$ (and the second equality holds because $U_1$ leaves one component in $H_2$ and $U_2$ leaves one component in $H[V(H_1) \cup \{w\}]$).
This yields a contradiction.
\end{proof}

The idea of focusing on blocks with whiskers to study Conjecture \ref{conjecture} was also used in \cite[Theorem 1.2]{SS22} and \cite[Corollary 3.13]{SS22b}. The latter result and \cite[Proposition 3]{LMRR21} imply that it is enough to prove Conjecture \ref{conjecture} for blocks with whiskers.

\begin{remark} \label{R.ReductionFewVertices}
The proof of Theorem \ref{T.minus_v_unmixed} also allows to show that if every accessible block with whiskers $G$ with at most $m$ vertices has a cut vertex $v$ such that $J_{G\setminus \{v\}}$ is unmixed, then every accessible graph with at most $m$ vertices has a strongly unmixed binomial edge ideal. In fact, if $H$ and $\B$ are as in the proof above, we have $|V(\B)| \leq |V(H)| \leq m$.
\end{remark}

\subsection{Accessible blocks with whiskers}

In the rest of the paper we use the following notation:

\begin{notation}
We denote by $\B$ a block $B$ with a whisker attached to each of the vertices $v_1, \dots, v_k \in V(B)$ and call $\B$ a \textit{block with $k$ whiskers}. We denote by $f_i$ the leaf adjacent to $v_i$ for every $i= 1 \dots, k$. We also assume that $B$ has $n$ vertices, hence $\B$ has $n+k$ vertices.
\end{notation}

In this subsection we may always assume that $k \geq 3$. In fact, if $k = 1,2$, the induced subgraph on the cut vertices is complete and by \cite[Proposition 6.6]{BMS22} there exists a cut vertex $v$ of $\B$ such that $J_{\B \setminus \{v\}}$ is unmixed.

In what follows we find several conditions on $\B$ that guarantee the existence of a cut vertex $v$ for which $J_{\B \setminus \{v\}}$ is unmixed, which we later use to prove the equivalence of the conditions in \eqref{Conditions} when $k \geq n-2$ and for every graph with at most $12$ vertices. These results will also play an important role in the next section.

\begin{proposition}\label{P.blockWithFreeVertex}
Let $\B$ be accessible. If $B$ has a free vertex $w$, then there exists a cut vertex $v$ of $\B$ such that $J_{\B \setminus \{v\}}$ is unmixed.
\end{proposition}

\begin{proof}
By \cite[Theorem 4.12]{BMS22} $w$ is adjacent to some cut vertex $v$ of $\B$ (since $k \geq 3$). By contradiction $J_{\B \setminus \{v\}}$ is not unmixed. Since $J_{\B}$ is unmixed, by Proposition \ref{P.good_cut_vertex} there exists a cut set $S$ of $\B \setminus \{v\}$ containing $N_B(v)$.

If $w$ is not a cut vertex of $\B$, then it is free also in $\B \setminus \{v\}$. Since $w \in S$ and cut sets only contain non-free vertices by \cite[Proposition 2.1]{RR14}, we have that $S \notin \CC(\B \setminus \{v\})$, a contradiction.

If $w$ is a cut vertex of $\B$, since it is free in $B$, then $\{w\} \cup N_B(w)$ is a clique. It follows that $N_B(w)\setminus \{v\} \subseteq N_B(v) \subseteq S$. Thus, $c_{\B\setminus \{v\}}(S)=c_{\B\setminus \{v\}}(S\setminus \{w\})$ and this yields a contradiction because $S$ is a cut set of $\B\setminus \{v\}$.
\end{proof}

\begin{lemma}\label{L.NeighboursOnlyCutVertices}
Assume that $J_{\B}$ is unmixed and let $v$ be a cut vertex of $\B$. If $N_B(v)$ contains only cut vertices of $\B$, then $J_{\B \setminus \{v\}}$ is unmixed.
\end{lemma}

\begin{proof}
By assumption, $N_B(v)$ is a cut set of $\B$ because every vertex in $N_B(v)$ is adjacent to $v$ and to a leaf. Since $J_{\B}$ is unmixed, the connected components of $\B \setminus N_B(v)$ are $|N_B(v)|$ isolated vertices and the whisker containing $v$. This implies that $N_B[v]=V(B)$. Note that $N_B(v) \notin \CC(\B \setminus \{v\})$, and every vertex of $(\B \setminus \{v\})\setminus N_B(v)=\B \setminus V(B)$ is isolated. Thus there is no cut set of $\B \setminus \{v\}$ containing $N_B(v)$, and by Proposition \ref{P.good_cut_vertex} it follows that $J_{\B \setminus \{v\}}$ is unmixed.
\end{proof}

Given a graph $G$ and $S \subset V(G)$, when we say that a vertex $v \in S$ \textit{reconnects} some connected components $G_1, \dots, G_r$ of $G \setminus S$, we mean that if we add back $v$ to $G \setminus S$, together with all edges of $G$ incident in $v$, then $G_1, \dots, G_r$ are in the same connected component.

\begin{lemma}\label{L.vertex_deg_2}
Let $\B$ be accessible and assume that there is a cut vertex $v$ of $\B$ such that $|N_B(v)|=2$. Then there exists a cut vertex $v'$ of $\B$ for which $J_{\B\setminus \{v'\}}$ is unmixed.
\end{lemma}

\begin{proof}
Let $N_B(v)=\{v_1,w\}$, where by \cite[Theorem 4.12]{BMS22} and Lemma \ref{L.NeighboursOnlyCutVertices} we may assume that $v_1$ is a cut vertex (since $k \geq 3$), while $w$ is not. Assume by contradiction that neither $J_{\B \setminus \{v\}}$ nor $J_{\B \setminus \{v_1\}}$ are unmixed. Hence, \cite[Lemma 6.2 (1)]{BMS22} says that $N_{B}(v) \nsubseteq N_{B}[v_1]$ and then $w \notin N_B(v_1)$.

Consider $N = N_B(v) \cup N_B(v_1) = N_B[v_1] \cup \{w\}$ and $N' = \{u \in N \mid N_{\B}(u) \nsubseteq N\} \setminus \{v,v_1\}$. By construction $N'$ is a cut set of $\B$ and we call $G_1, \dots, G_{|N'|+1}$ the connected components of $\B \setminus N'$, where $v \in G_1$. Clearly $N' \subseteq (N_B(v_1) \setminus \{v\}) \cup \{w\}$. We show the reverse inclusion.

By Proposition \ref{P.good_cut_vertex}, $N_B(v_1)$ is contained in a cut set of $\B\setminus\{v_1\}$ and, in particular, every vertex in $N_B(v_1)$ is adjacent to at least two vertices of $\B \setminus N_B[v_1]$. Since $N \setminus N_B[v_1]=\{w\}$, it follows that every vertex $u \in N_B(v_1)\setminus \{v\}$ is in $N'$. We claim that also $w \in N'$. Indeed, if $w \notin N'$, then $G_1 \setminus \{v\}$ has at least three connected components: the one containing $w$, the one containing $v_1$, and an isolated vertex. Therefore, $c_{\B}(N' \cup \{v\}) \geq |N'|+3 > |N' \cup \{v\}| + 1$, which contradicts Lemma \ref{L.CutsetsInUnmixedGraphs}. Hence, $N' = (N_B(v_1) \setminus \{v\}) \cup \{w\} \in \CC(\B)$.

Since $v \in G_1$, the vertices of $G_1$ are $v, v_1$ and the two leaves adjacent to them. Both $G_1\setminus \{v\}$ and $G_1\setminus \{v_1\}$ have two connected components, hence Remark \ref{R.union} implies that both $N' \cup \{v\}$ and $N' \cup \{v_1\}$ are cut sets of $\B$. Moreover, the accessibility of $\B$ implies that there exists $u\in N'$ such that $N' \setminus \{u\}$ is a cut set of $\B$. It follows from the unmixedness of $J_{\B}$ that $u$ is adjacent to exactly two connected components of $\B \setminus N'$.

If $u \in N_B(v_1) \setminus \{v\}$, then $u$ does not reconnect any component of $\B \setminus (N' \cup \{v_1\})$, and this contradicts the fact that $N' \cup \{v_1\} \in \CC(\B)$. The same argument works if $u=w$, replacing $v_1$ with $v$.
\end{proof}

The next result allows us to assume that every vertex of $B$ has at least degree three in $B$. This fact will be of great importance in the algorithm that we will develop in the next section.

\begin{proposition}\label{P.blocksWithDegree2verticesAreGood}
Let $\B$ be accessible. If there is a vertex of $B$ that is adjacent to at most two vertices of $B$, then there exists a cut vertex $v$ of $B$ such that $J_{\B\setminus \{v\}}$ is unmixed.
\end{proposition}

\begin{proof}
We may assume that $|V(B)|>2$ and there are no vertices of degree one in a block with at least three vertices. Assume that $w \in V(B)$ has degree two in $B$. If $w$ is a cut vertex of $\B$, the claim follows from Lemma \ref{L.vertex_deg_2}. Suppose $w$ is not a cut vertex of $\B$; then $w$ is adjacent to a cut vertex $v$ of $\B$ by \cite[Theorem 4.12]{BMS22}. If $J_{\B\setminus \{v\}}$ is not unmixed, i.e., $N_B(v)$ is contained in a cut set $S$ of $\B\setminus \{v\}$ by Proposition \ref{P.good_cut_vertex}, we get a contradiction because $w \in S$ and it is adjacent to exactly one vertex of $\B \setminus \{v\}$ and $S$ would not be a cut set.
\end{proof}

We now prove two lemmas that will be very useful in the rest of the section.

\begin{lemma} \label{L.strictly}
Let $J_{\B}$ be unmixed and $v$ be a cut vertex of $\B$. Suppose that $J_{\B \setminus \{v\}}$ is not unmixed, i.e., there exists $S \in \CC(\B \setminus \{v\})$ such that $N_B(v) \subseteq S$. If $S \setminus \{s\} \in \CC(\B)$ for some $s \in V(\B)$, then $s \notin N_B(v)$. In particular, if $\B$ is accessible, then $N_B(v) \subsetneq S$.
\end{lemma}

\begin{proof}
Assume by contradiction that $s \in N_B(v)$. Since $S$ is a cut set of $\B \setminus \{v\}$, $s$ reconnects at least two connected components of $\B \setminus \{v\}$. Moreover, $s$ is also adjacent to $v$ and then in $\B \setminus S$ it reconnects at least three connected components; this is a contradiction because $|S|=c_{\B}(S\setminus \{s\})\leq c_{\B}(S)-2=|S|+1-2=|S|-1$.

In particular, if $\B$ is accessible, then there exists $s \in S$ such that $S \setminus \{s\}$ is a cut set of $\B$ and $s \in S \setminus N_B(v)$.
\end{proof}

In the next lemma we use an argument similar to the proof of \cite[Proposition 6.6]{BMS22}.

\begin{lemma} \label{L.cvdeg1}
Let $\B$ be accessible and suppose that there exists $w \in V(B)\setminus \{v_1, \dots, v_k\}$ that is adjacent to $v_j$ for exactly one $j \in \{1, \dots, k\}$. Assume also that this $v_j$ is not a cut vertex of the graph induced by $\B$ on $\{v_1, \dots, v_k\}$. Then $J_{\B \setminus \{v_j\}}$ is unmixed.
\end{lemma}

\begin{proof}
Without loss of generality, we may assume $j=1$. Let $N=N_B(v_2) \cup \dots \cup N_B(v_k)$, which is not empty because $k \geq 3$. Note that $w \in V(B) \setminus N$ and $N \cup N_B(v_1)=V(B)$ because every vertex of $\B$ is adjacent to a cut vertex by \cite[Theorem 4.12]{BMS22}. Consider $N'=\{u \in N \mid N_B(u) \nsubseteq N\}\subseteq N\setminus \{v_2, \dots, v_k\}$; it is easy to check that $N'$ is a cut set of $\B$.

Assume by contradiction that $J_{\B \setminus \{v_1\}}$ is not unmixed, i.e., by Proposition \ref{P.good_cut_vertex}, there exists $S \in \mathcal{C}(\B \setminus \{v_1\})$ such that $N_B(v_1) \subseteq S$.
Clearly, $S$ is also a cut set of $\B$ and by \cite[Corollary 4.15]{BMS22} we can order the elements of $S=\{z_1, \dots, z_t\}$ in such a way that $\{z_1, \dots, z_i\} \in \CC(\B)$ for every $i=1, \dots, t$.
Let $N_B(v_1) \setminus N=\{z_{i_1}, \dots, z_{i_s}\}$ with $i_1< \dots < i_s$; notice that $N_B(v_1) \setminus N \neq \emptyset$, because $w \in N_B(v_1) \setminus N$. Since $S \in \mathcal{C}(\B \setminus \{v_1\})$, $z_{i_s}$ is adjacent to at least two vertices $a_s, a_{s+1} \in V(B) \setminus (S \cup \{v_1\}) \subseteq V(B) \setminus N_B[v_1]\subseteq N \setminus \{v_1\}$, and both $a_s$ and $a_{s+1}$ are in $N'\setminus \{v_1\}$ because $z_{i_s}\notin N$.

Consider now $S'=\{z_1, \dots, z_{i_{s-1}}\} \in \mathcal{C}(\B)$. A connected component $C$ of $\B \setminus S'$ contains $z_{i_s}$, hence $N_{B}(z_{i_s}) \setminus S' \subseteq C$; in particular $v_1, a_s, a_{s+1} \in C$. Therefore, $z_{i_{s-1}}$ is adjacent to at least one vertex $a_{s-1} \in V(B) \setminus (N_B[v_1] \cup \{a_s,a_{s+1}\}) \subseteq N \setminus \{v_1,a_s, a_{s+1}\}$. Again, $a_{s-1} \in N' \setminus \{v_1,a_s,a_{s+1}\}$ because $z_{i_{s-1}} \notin N$. Repeating the same argument, we find $\{a_1, \dots, a_{s+1}\} \subseteq N'\setminus \{v_1\}$, where the $a_i$'s are pairwise distinct.

Moreover, $v_1 \in N'$ because it is adjacent to some $v_i$ and to $w \notin N$ by assumption. Recall also that $v_2, \dots, v_k \in N \setminus N'$.
Hence, $|N'| \geq s+2$ and the connected components of $\B \setminus N'$ are the following:
\begin{itemize}
\item the component containing $N\setminus N'$, which is connected because by assumption the graph induced by $\B$ on $\{v_{2}, \dots,v_k\}$ is connected since $v_1$ is not a cut vertex of the graph induced by $\B$ on $\{v_1, \dots,v_k\}$;
\item the isolated vertex that is adjacent to $v_1$ in $\B$;
\item the connected components of $N_B(v_1)\setminus N$, which are at most $s$ since $N_B(v_1)\setminus N$ has cardinality $s$.
\end{itemize}
Since $N'$ is a cut set of $\B$ and $J_{\B}$ is unmixed, we get $c_{\B}(N')\leq 1+1+s \leq |N'|=c_{\B}(N')-1$, which yields a contradiction.
\end{proof}

Most of the cases in the next proof follow from \cite[Proposition 6.6]{BMS22}. The only missing case has been independently obtained by Saha and Sengupta in \cite[Proposition 5.3]{SS22}; we include here a simpler and shorter proof.

\begin{proposition}\label{P.three_whiskers}
Assume that $\B$ is accessible and $k \leq 3$. Then there exists a cut vertex $v$ of $\B$ such that $J_{\B\setminus\{v\}}$ is unmixed.
\end{proposition}

\begin{proof}
We recall that the subgraph $H$ induced by $\B$ on the cut vertices of $\B$ is connected by \cite[Proposition 4.10]{BMS22}.
By \cite[Proposition 6.6]{BMS22}, it follows that the claim holds if $H$ is complete. Therefore, we only need to consider the case in which $V(H)=\{v_1,v_2,v_3\}$ and, without loss of generality, $E(H)=\{\{v_1,v_2\},\{v_2,v_3\}\}$.

If $J_{\B \setminus \{v_2\}}$ is unmixed, we have nothing to prove; if this is not the case, by Proposition \ref{P.good_cut_vertex} there exists $S \in \mathcal{C}(\B \setminus \{v_2\})$ such that $N_B(v_2) \subseteq S$.  If $r=|N_B(v_2)|$, then $|S| \geq r+1$ by Lemma \ref{L.strictly}. Recall that $f_1, f_2, f_3$ are the leaves attached to $v_1, v_2, v_3$ respectively. The connected components of $\B \setminus S$ are the isolated vertices $f_1$ and $f_3$, the edge $\{v_2, f_2\}$ and other $|S|-2$ connected components whose vertices are included in $V(B) \setminus (S \cup \{v_2\})$. This implies that $|V(B) \setminus (S \cup \{v_2\})| \geq |S|-2$ and, since $N_B(v_2) \subsetneq S$, there are at least $|S|-1$ vertices in $V(B) \setminus N_B[v_2]$. By Lemma \ref{L.cvdeg1}, we may assume that $V(B) \setminus N_{B}[v_2] \subseteq N_B(v_1) \cap N_B(v_2)$. In particular, since also $v_2 \in N_B(v_1)$, we have that $|N_B(v_1)| \geq |S| \geq r+1$.

Assume by contradiction that $N_B(v_1)$ is contained in a cut set of $\B\setminus \{v_1\}$; in particular, by \cite[Remark 5.4]{BMS22} $N_B(v_1)$ is a cut set of $\B$. The connected components of $\B \setminus N_B(v_1)$ are the edge $\{v_1,f_1\}$, the isolated vertex $f_2$ and other $|N_B(v_1)|-1 \geq r$ connected components whose vertices (except $f_3$ which is adjacent to $v_3$) are included in $V(B) \setminus N_B[v_1]$; thus $|V(B) \setminus N_B[v_1]|  \geq r$. However, $V(B) \setminus N_B[v_1] \subseteq N_B(v_2) \setminus \{v_1\}$ and hence $r \leq |V(B) \setminus N_B[v_1]| \leq |N_B(v_2) \setminus \{v_1\}| \leq r-1$ yields a contradiction.
\end{proof}

Before treating blocks with four whiskers we prove the following result that is interesting by itself.

\begin{proposition} \label{P.BoundNeighbourhood}
Assume that $\B$ is accessible and let $H$ be the graph induced by $\B$ on $\{v_1,\ldots,v_k\}$. Assume that there are no cut vertices $w$ of $\B$ such that $J_{\B \setminus \{w\}}$ is unmixed. If $v$ is a cut vertex of $\B$ and $r=|N_H[v]|$, then $|N_B(v)| \leq \lfloor \frac{n+r}{2} \rfloor - 2$. In particular, $|N_B(v)| \leq \lfloor \frac{n+k}{2} \rfloor - 2$.
\end{proposition}

\begin{proof}
By Proposition \ref{P.good_cut_vertex}, there exists $S \in \mathcal{C}(\B \setminus \{v\})$ containing $N_B(v)$. Assume by contradiction that $|N_B(v)| \geq \lfloor \frac{n+r}{2} \rfloor - 1$.
Since $S$ is also a cut set of $\B$ and $J_{\B}$ is unmixed, $c_{\B}(S) = |S|+1$. Let $v_1, \dots, v_{k-r}$ be the cut vertices of $\B$ different from $v$ that are not adjacent to $v$.
Assume that $v_1, \dots, v_i \in S$ and $v_{i+1}, \dots, v_{k-r} \notin S$ with $0 \leq i \leq k-r$.

Suppose first that $S=N_B(v) \cup \{v_1, \dots, v_i\}$. By \cite[Remark 5.4]{BMS22}, $N_B(v)$ is a cut set of $\B$. Since also $S \in \mathcal{C}(\B)$, $c_{\B}(N_B(v))=c_{\B}(S)-i$; more precisely, this means that every $v_j$ reconnects the leaf $f_j$ to only another component. This easily implies that also $N_B(v)$ is a cut set of $\B \setminus \{v\}$, but this contradicts Lemma \ref{L.strictly}.

Therefore, we may assume that $|S| \geq |N_{B}(v)|+i+1$. Note that in $\B \setminus S$ we have that $f_j$ is in the same connected component of $v_j$, for $i+1 \leq j \leq k-r$ and $f$ is in the same connected component of $v$, where $f$ is the leaf attached to $v$. Then, we have $|S| + c_{\B}(S) +(k-r-i)+1 \leq n+k$. Thus,
   \begin{align*}
    n+k &\geq |S| + c_{\B}(S)+ k-r-i + 1 = 2|S| +k-r-i+ 2 \geq 2(|N_B(v)|+i+1) + k-r-i+ 2 \geq \\
   & \geq 2 \left( \bigg\lfloor \frac{n+r}{2} \bigg\rfloor - 1 \right) +k-r+i+ 4  \geq (n+r-1)-2+k-r+i+4\geq n+k+1,
    \end{align*}
    which is a contradiction.
\end{proof}

\begin{remark}\label{R.boundsNumEdgesBlocks}
Let $\B$ be accessible and suppose that for every cut vertex $v$ of $\B$ the ideal $J_{\B \setminus \{v\}}$ is not unmixed. By Proposition \ref{P.blocksWithDegree2verticesAreGood}, for every vertex $v$ of $B$ we have $|N_B(v)| \geq 3$, and thus $|E(B)| \geq \frac{3n}{2}$. On the other hand, Proposition \ref{P.BoundNeighbourhood} implies that for all the cut vertices $v$ of $\B$ we have $|N_B(v)| \leq \lfloor \frac{n+k}{2} \rfloor - 2$. Let us consider the non-cut vertices of $\B$ that are not leaves.

Notice that at most one of these vertices can have degree $n-1 = |V(B)|-1$. In fact, if $w_1,w_2$ are two such vertices and they both have degree $n-1$, then $N_B(w_1) = V(B) \setminus \{w_1\}$ and $N_B(w_2) = V(B) \setminus \{w_2\}$. By Proposition \ref{P.good_cut_vertex} and Lemma \ref{L.strictly}, there exists $S \in \mathcal{C}(\B \setminus \{v\})$ such that $N_B(v) \subsetneq S$. In particular, $S$ is a cut set of $\B$ and $w_1,w_2 \in S$. Since $\B$ is accessible, there exists a cut set $T \subsetneq S$ such that $w_1 \notin T$ and $w_2 \in T$. However, $w_2$ does not reconnect anything in $\B \setminus T$ because $N_B[w_1] = V(B)$, against $T$ being a cut set of $\B$. It follows that
\[
\sum_{\stackrel{w \in V(B)}{w \text{ non-cut vertex of } \B}} |N_B(w)| \leq (n-1) + (n-2)(n-k-1).
\]

Putting everything together, we conclude that, if $\B$ is accessible and for every cut vertex $v$ of $\B$ the ideal $J_{\B \setminus \{v\}}$ is not unmixed, then:
\begin{equation}\label{eq.boundsNumEdgesBlock}
\frac{3n}{2} \leq |E(B)| \leq \frac{1}{2} \left( k \left( \left\lfloor \frac{n+k}{2} \right\rfloor - 2 \right) + (n-1) + (n-2)(n-k-1) \right)=\frac{(n-1)^2}{2}-\frac{k}{2} \left( n - \left\lfloor \frac{n+k}{2} \right\rfloor \right).
\end{equation}
\end{remark}

Since $n - \left\lfloor \frac{n+k}{2} \right\rfloor \geq 0$, the upper bound is non-trivial, i.e., the right-hand side of (\ref{eq.boundsNumEdgesBlock}) is strictly less than $\binom{n}{2}$.

\begin{lemma} \label{L.BoundEdgesH}
Let $\B$ be accessible and $H$ be the graph induced by $\B$ on its cut vertices. Assume that $H$ is a block such that $|E(H)| \geq \frac{1}{2}\sum_{i=1}^k \lfloor \frac{n+r_i}{2} \rfloor -n+1$, where $r_i:=|N_H[v_i]|$. Then there exists a cut vertex $v$ of $\B$ such that $J_{\B \setminus \{v\}}$ is unmixed.
\end{lemma}

\begin{proof}
Suppose by contradiction that there are no cut vertices of $\B$ such that $J_{\B \setminus \{v\}}$ is unmixed. Since $H$ is a block, by Lemma \ref{L.cvdeg1} we may assume that for every vertex $w$ in $V(B) \setminus V(H)$ there exist indices $i \neq j$ such that $\{w,v_i\}, \{w,v_j\} \in E(B)$. Thus there exist at least $2|V(B)\setminus V(H)|=2(n-k)$ edges of $B$ with exactly one endpoint in $V(H)$.

By Proposition \ref{P.BoundNeighbourhood}, we have that $|N_{B}(v_i)| \leq \lfloor \frac{n+r_i}{2} \rfloor -2$ for every $i$. Therefore, by our assumption, the number of edges of $B$ having exactly one endpoint in $V(H)$ is at most
$$\sum_{i=1}^k \left\lfloor \frac{n+r_i}{2} \right\rfloor -2k-2|E(H)|\leq \sum_{i=1}^k \left\lfloor \frac{n+r_i}{2} \right\rfloor -2k-\sum_{i=1}^k \left\lfloor \frac{n+r_i}{2} \right\rfloor+2n-2=2(n-k)-2.$$
This yields a contradiction since $2(n-k)>2(n-k)-2$.
\end{proof}

The technical inequality in Lemma \ref{L.BoundEdgesH} is always satisfied when $k=4$, as we show next.

\begin{proposition} \label{P.FourWhiskers}
Let $\B$ be accessible with $k=4$ whiskers. If the graph $H$ induced by $\B$ on its cut vertices is a block, then there exists a cut vertex $v$ of $\B$ such that $J_{\B \setminus \{v\}}$ is unmixed.
\end{proposition}

\begin{proof}
Let $h=\mathrm{max} \{r_i: 1 \leq i \leq 4\}$, where $r_i$ is defined as in Lemma \ref{L.BoundEdgesH}. It is enough to show that $|E(H)| \geq h+1$ for every block $H$ on $4$ vertices. Indeed, $h+1 = \frac{1}{2}\sum_{i=1}^4 \frac{n+h}{2} -n+1 \geq \frac{1}{2}\sum_{i=1}^4 \lfloor \frac{n+r_i}{2} \rfloor -n+1$ and Lemma \ref{L.BoundEdgesH} would imply the claim. If $H$ is a square, then $|E(H)|=4$ and $h=3$, whereas if $H$ is not a square, $|E(H)| \geq 5$ and $h=4$. In both cases we conclude that $|E(H)| \geq h+1$.
\end{proof}

\subsection{New classes of graphs satisfying the conjecture}

Our next goal is to prove that the four conditions in \eqref{Conditions} are equivalent when $k \geq n-2$.

\begin{remark} \label{R.k=n}
If $J_{\B}$ is unmixed and $k=n$, we claim that $B$ is complete. Therefore, in this case $J_{\B}$ is strongly unmixed by applying \cite[Lemma 12]{LMRR} $k$ times and, in particular, there exists a cut vertex $v$ of $\B$ such that $J_{\B \setminus \{v\}}$ is unmixed.

To show the claim, suppose that $B$ is not complete. Notice that $k=n$ means that every vertex in $B$ is a cut vertex of $\B$, so we may assume that $\{v_1,v_2\} \notin E(\B)$. Moreover, $N_B(v_1)$ is a cut set of $\B$ because every vertex in $N_B(v_1)$ reconnects the leaf adjacent to it to $v_1$. However, $c_{\B}(N_B(v_1)) \geq |N_B(v_1)| + 2$ counting $|N_B(v_1)|$ leaves, the component containing $v_1$ and the component containing $v_2$. This contradicts the unmixedness of $J_{\B}$.
\end{remark}

\begin{proposition}\label{P.blocksWithTwoNonCutvertex}
If $k \geq n-2$ and $\B$ is accessible, then there exists a cut vertex $v$ of $\B$ such that $J_{\B \setminus \{v\}}$ is unmixed.
\end{proposition}

\begin{proof}
The case $k=n$ follows from Remark \ref{R.k=n}.
Assume first that $k=n-1$. This means that in $B$ there is exactly one vertex $w$ that is not a cut vertex of $\B$. Let $v \in B$ be a cut vertex of $\B$ and assume by contradiction that there exists a cut set $S$ of $\B \setminus \{v\}$ c, ontaining $N_B(v)$. If $w \not\in N_B(v)$, we conclude by Lemma \ref{L.NeighboursOnlyCutVertices}. Then we can assume that $w \in N_B(v) \subseteq S$ and every vertex in $S \setminus \{w\}$ is a cut vertex. Hence, \cite[Proposition 4.18]{BMS22} implies that $S \setminus \{w\}$ is a cut set of $\B$ and this contradicts Lemma \ref{L.strictly}.

Finally, assume that $k=n-2$, i.e., in $B$ there are exactly two vertices $w_1$, $w_2$ that are not cut vertices of $\B$. Suppose by contradiction that $J_{\B\setminus \{v_i\}}$ is not unmixed, for every $1 \leq i \leq k$.
As above, we may assume that every cut vertex of $\B$ is adjacent to at least one between $w_1$ and $w_2$, and we claim that every $v_i$ is adjacent to exactly one of them.
Fix $i \in \{1, \dots, k\}$. By Proposition \ref{P.good_cut_vertex} there exists $S_i \in \CC(\B)$ containing $N_B(v_i)$ and we may assume that $v_i$ is adjacent to $w_1$.
By \cite[Proposition 4.18]{BMS22} either $S_i \setminus \{w_1\}$ or $S_i \setminus \{w_2\}$ is a cut set of $\B$. Since $w_1 \in N_B(v_i)$, Lemma \ref{L.strictly} implies that $w_2 \in S_i \setminus N_B(v_i)$, in particular $v_i$ is not adjacent to $w_2$.

Let $v_1, \dots, v_a$ and $v_{a+1}, \dots, v_k$ be the cut vertices adjacent to $w_1$ and $w_2$ respectively. Since $S_1$ and $S_1 \setminus \{w_2\}$ are cut sets of $\B$, by \cite[Lemma 4.14]{BMS22} $w_2$ is adjacent to at least two cut vertices, say $v_{k-1}$ and $v_{k}$, not contained in $S_1$ and such that $\{v_{k-1},v_k\} \notin E(\B)$.
Consider the set $T=\{v_2, \dots, v_{k-1}, w_2\}$ and recall that, for every $i$, $f_i$ is the leaf adjacent to $v_i$. The connected components of $\B\setminus T$ are a path on $\{f_1, v_1, w_1\}$, the edge $\{v_k,f_k\}$, and the isolated vertices $f_2, \dots, f_{k-1}$. Therefore, $T$ is a cut set of $\B$ by Lemma \ref{L.CutsetsInUnmixedGraphs}. On the other hand, since $v_{k-1}$ is not adjacent to $v_1$, $w_1$ and $v_k$, it follows that $c_{\B}(T)=c_{\B}(T\setminus \{v_{k-1}\})$, which yields a contradiction.
\end{proof}

In the case $k \geq n-2$ we also prove that accessibility, strong unmixedness, and hence Cohen-Macaulayness, are equivalent.

\begin{theorem} \label{T.Conjecture_n=k+2}
If $k \geq n-2$ and $\B$ is accessible, then $J_{\B}$ is strongly unmixed. In particular, Conjecture \ref{conjecture} holds for $\B$.
\end{theorem}

\begin{proof}
We proceed by induction on $k$. If $k=1$, then $n \geq 3$ and $B$ is either a path of length at most two or a triangle; in both cases $J_{\B}$ is strongly unmixed.
Assume $k>1$. By Proposition \ref{P.blocksWithTwoNonCutvertex}, there exists a cut vertex, say $v_1$, such that $J_{\B\setminus \{v_1\}}$ is unmixed, and by \cite[Corollary 5.16]{BMS22} the graphs $\B\setminus \{v_1\}$ and $\B_{v_1}\setminus \{v_1\}$ are accessible.
The graph $\B\setminus \{v_1\}$ consists of an isolated vertex and another connected component $A$. Assume first that $A$ is a block with $k$ whiskers. Therefore, $|V(A)| \geq 2k$ vertices, but on the other hand $|V(A)|=n+k-2\leq 2k$, which implies $|V(A)|=2k$. By Remark \ref{R.k=n}, $J_A$ is strongly unmixed and then $J_{\B\setminus \{v_1\}}$ is strongly unmixed.
If $A$ is not a block with $k$ whiskers, it is easy to see that every block $C$ of $\B \setminus \{v_1\}$ has at most $k-1$ cut vertices. Since $\C$ is accessible by \cite[Proposition 3]{LMRR21}, by induction $J_{\C}$ is strongly unmixed.
Moreover, by \cite[Theorem 3.17]{SS22} it follows that $J_{\B \setminus \{v_1\}}$ is strongly unmixed. Therefore, in light of Proposition \ref{P.stronglyUnmixed}, we only need to prove that $J_{\B_{v_1} \setminus \{v_1\}}$ is strongly unmixed.

Let $G=\B_{v_1} \setminus \{v_1\}$ and $H=G\setminus \{f_1\}$.
We claim that $H$ is accessible. We first note that there are no cut sets of $H$ containing $N_G(f_1)=N_{B}(v_1)$. Indeed, if there is such a cut set $S$, then $S$ would be a cut set of $\B \setminus \{v_1\}$ containing $N_{B}(v_1)$ and this is not possible because $J_{\B \setminus \{v_1\}}$ is unmixed.
It is clear that a cut set $T$ of $H$ is also a cut set of $G$, because $f_1$ is a free vertex of $G$ and $H=G\setminus \{f_1\}$. Moreover, $c_H(T)=c_G(T)=|T|+1$ because $T$ does not contain $N_G(f_1)$. Finally, since $G$ is accessible, there exists $t \in T$ such that $T \setminus \{t\}$ is a cut set of $G$ and it easily follows that $T \setminus \{t\}$ is also a cut set of $H$; hence, $H$ is accessible.

As in the case of $\B \setminus \{v_1\}$, the induction hypothesis together with \cite[Proposition 3]{LMRR21} and \cite[Theorem 3.17]{SS22} imply that $J_{G}$ is strongly unmixed.
\end{proof}

However, if $k \geq n-1$, the strong unmixedness of $J_{\B}$ is equivalent to its unmixedness.

\begin{proposition}\label{P.blocksWithOneNonCutvertex}
Let $k \geq n-1$. If $J_{\B}$ is unmixed, then it is strongly unmixed. In particular, $R/J_{\B}$ is Cohen-Macaulay if and only if $J_{\B}$ is unmixed.
\end{proposition}

\begin{proof}
The case $k=n$ follows from Remark \ref{R.k=n}, then assume $k=n-1$. By Theorem \ref{T.Conjecture_n=k+2} it is enough to prove that $\B$ is accessible. Let $w$ be the unique non cut vertex of $\B$ contained in $B$, and let $S$ be a non-empty cut set of $\B$.
If $w \notin S$, $S$ consists only of cut vertices and hence $S \setminus \{v_i\}$ is a cut set of $\B$ for every $v_i \in S$ by \cite[Lemma 4.1]{BMS22}. 

Now suppose that $w \in S$. Any cut vertex $v_j$ of $S$ reconnects the leaf $f_j$ to some other vertex in $B$ and this is still the case in $\B \setminus (S \setminus \{w\})$. This shows that $S \setminus \{w\}$ is a cut set of $\B$. Thus, $\B$ is accessible.
\end{proof}

\begin{example} \label{E.k=n-2}
Notice that Proposition \ref{P.blocksWithOneNonCutvertex} does not hold if $k=n-2$. For instance, let $G$ be the graph in Figure \ref{F.unmixedNonAccessible}, which has $7$ vertices in the block and $5$ whiskers. One can check that $J_G$ is unmixed and not accessible: in fact, $S = \{3,4,6,7\} \in \CC(G)$ but none of its subsets of cardinality three is a cut set.
\begin{figure}[ht!]
\centering
\begin{tikzpicture}[scale=0.65]
\node[label={left:{\small $1$}}] (a) at (0,0) {};
\node[label={right:{\small $2$}}] (b) at (2,0) {};
\node[label={below:{\small $3$}}] (c) at (3.2469796037174667,1.5636629649360592) {};
\node[label={above right:{\small $4$}}] (d) at (2.801937735804838,3.513518789299706) {};
\node[label={below:{\small $5$}}] (e) at (1,4.381286267534822) {};
\node[label={above left:{\small $6$}}] (f) at (-0.8019377358048378,3.513518789299707) {};
\node[label={below:{\small $7$}}] (g) at (-1.2469796037174667,1.56366296493606) {};
\node[label={below left:{\small $8$}}] (h) at (-0.8677674782351161,-1.801937735804838) {};
\node[label={below right:{\small $9$}}] (i) at (2.8677674782351135,-1.8019377358048392) {};
\node[label={below right:{\small $10$}}] (j) at (5.196835428081113,1.118621097023431) {};
\node[label={above:{\small $11$}}] (k) at (1,6.381286267534821) {};
\node[label={below left:{\small $12$}}] (l) at (-3.196835428081114,1.1186210970234303) {};
\draw (h) -- (a) -- (g) -- (f) -- (e) -- (d) -- (c) -- (b) -- (i)
(l) -- (g) -- (b) -- (d) -- (f) -- (a) -- (c) -- (j)
(c) -- (e) -- (g)
(a) -- (d)
(b) -- (f)
(e) -- (k);
\end{tikzpicture}
\caption{A non-accessible graph $G$ such that $J_G$ is unmixed}\label{F.unmixedNonAccessible}
\end{figure}
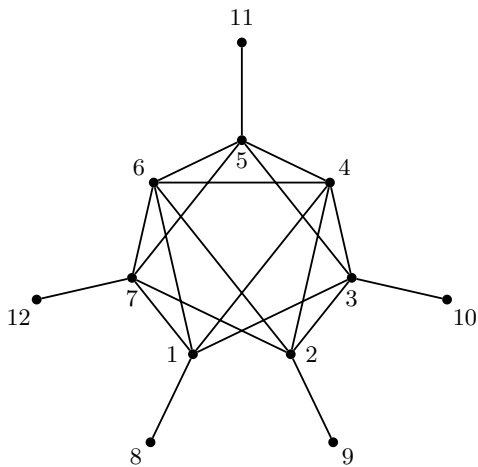
\end{example}

We are ready to prove that Conjecture \ref{conjecture} holds for every graph with at most $12$ vertices. This was verified computationally in \cite{LMRR21} by an exhaustive search on all connected graphs that took about a month on a high-performance computing system.

\begin{theorem}\label{T.12_vertices}
If $G$ is an accessible graph with at most 12 vertices, then $J_G$ is strongly unmixed. In particular, Conjecture \ref{conjecture} holds true for $G$.
\end{theorem}

\begin{proof}
By Remark \ref{R.ReductionFewVertices} it is enough to prove that in every block with whiskers $\B$ with at most $12$ vertices there is a cut vertex $v$ such that $J_{\B \setminus \{v\}}$ is unmixed, i.e., $N_B(v)$ is not contained in any cut set of $\B \setminus \{v\}$. Assume by contradiction that this is not the case.
By Proposition \ref{P.three_whiskers}, $\B$ has at least $4$ whiskers. Moreover, Proposition \ref{P.blocksWithTwoNonCutvertex} implies that $\B$ has exactly $4$ whiskers and at least $11$ vertices.

By Proposition \ref{P.FourWhiskers} we know that the graph $H$ induced on the cut vertices $v_1, v_2, v_3, v_4$ is not a block. Then $H$ has a leaf, say $v_1$, and we may assume $N_H(v_1)=\{v_2\}$. Note that $|N_B(v_1)| >2$ by Proposition \ref{P.blocksWithDegree2verticesAreGood}; moreover, if $\B$ has 11 vertices, Proposition \ref{P.BoundNeighbourhood} ensures that $|N_B(v_1)| \leq 2$, giving a contradiction.

Hence, we assume that $\B$ has $12$ vertices. Its vertices are the cut vertices $v_1, v_2,v_3,v_4$, their leaves $f_1, f_2, f_3,f_4$ and four more vertices $w_1, w_2, w_3, w_4$, which are not cut vertices of $\B$. Proposition \ref{P.blocksWithDegree2verticesAreGood} ensures that $|N_B(v_i)| \geq 3$ and $|N_B(w_i)| \geq 3$, for every $1 \leq i \leq 4$. By Proposition \ref{P.BoundNeighbourhood} we have that $|N_B(v_1)|=3$, say $N_B(v_1)=\{v_2, w_1, w_2\}$. By our assumption and Proposition \ref{P.good_cut_vertex}, there exists $S \in \mathcal{C}(\B \setminus \{v_1\})$ containing $N_B(v_1)$. By \cite[Proposition 4.18]{BMS22}, it follows that $S\setminus \{w_i\}$ is a cut set of $\B$ for some $i$, and Lemma \ref{L.strictly} implies that $w_i \notin N_B(v_1)$. We may assume $i=3$ and then $\{v_2, w_1, w_2, w_3\} \subseteq S$. Hence, $c_{\B}(S)\geq 5$ and it is straightforward to see that this is possible only if $S=\{v_2, w_1, w_2, w_3\}$ and the vertices of the connected components of $\B \setminus S$ are $\{f_2\}$, $\{v_1, f_1\}$, $\{v_3, f_3\}$, $\{v_4, f_4\}$, and $\{w_4\}$. In particular, $v_3$ and $v_4$ are not adjacent.

Since $H$ is not a block, it can only be a path, a triangle with a whisker or a star, but in the first two cases $v_3$ and $v_4$ would be adjacent, hence $H$ is a star with center in $v_2$ and leaves $v_1,v_3,v_4$. Moreover, $w_4$ is not adjacent to $v_1$, $v_3$ and $v_4$, thus it is adjacent to $v_2$ by \cite[Theorem 4.12]{BMS22} and to at least one between $w_1$ and $w_2$. By Lemma \ref{L.NeighboursOnlyCutVertices} and Proposition \ref{P.BoundNeighbourhood}, it follows that $|N_B(v_2)|=4$ and $N_B(v_2)=\{v_1,v_3,v_4,w_4\}$. By Lemma \ref{L.cvdeg1}, $\{v_3,v_4\} \subseteq N_B(w_3)$, $w_1 \in N_B(v_3) \cup N_B(v_4)$ and $w_2 \in N_B(v_3) \cup N_B(v_4)$. This fact and Proposition \ref{P.BoundNeighbourhood} imply that $|N_B(v_3)|=|N_B(v_4)|=3$, then we may assume $N_B(v_3)=\{v_2,w_1,w_3\}$ and $N_B(v_4)=\{v_2,w_2,w_3\}$. Notice that $N_B(v_3)$ and $N_B(v_4)$ are cut sets of $\B$ by \cite[Remark 5.4]{BMS22}. If $w_4$ is adjacent to $w_2$, in $\B \setminus N_B(v_3)$ there are three connected components, $\{v_1,v_4,f_1,f_4,w_2,w_4\}$, $\{f_2\}$ and $\{v_3,f_3\}$; analogously, if $w_4$ is adjacent to $w_1$, in $\B \setminus N_B(v_4)$ the three connected components are $\{v_1,v_3,f_1,f_3,w_1,w_4\}$, $\{f_2\}$ and $\{v_4,f_4\}$. In both cases, we contradict the unmixedness of $J_{\B}$.
\end{proof}

\section{The algorithm}\label{S.algorithm}

In this section, in order to answer Question \ref{Q.Main_question}, we describe an algorithm that employs Theorem \ref{T.filters} to restrict the search space to certain blocks with whiskers. The key idea of the algorithm is to generate all blocks with $n$ vertices and then add $k$ whiskers to suitable subsets of the vertices of each block. In this way we extend the computations in \cite{LMRR} because we can examine graphs with $n+k$ vertices generating only blocks with $n$ vertices, limiting the rapid growth in the number of graphs.
This allows us to verify Conjecture \ref{conjecture} for many new graphs.

We start with a lemma that describes basic properties of a block with whiskers.

\begin{lemma}\label{L.unmixedBlocksWithWhiskers}
If $J_{\B}$ is unmixed, then the following conditions are satisfied:
\begin{itemize}
  \item[(a)] for every $S \subseteq \{v_1,\dots,v_k\}$, $B \setminus S$ is connected;
  \item[(b)] for every $S \in \CC(B)$ containing $k_S \geq 0$ cut vertices of $\B$, $c_B(S) = |S| - k_S + 1$.
\end{itemize}
Moreover, if $k < n$, then
\[
\{S : S \subseteq \{v_1,\dots,v_k\} \} \cup \{S : S \in \CC(B) \} \subseteq \CC(\B).
\]
\end{lemma}

\begin{proof}
Suppose first $k=n$. If $B$ is complete, there is nothing to prove. Otherwise, there exists a non-empty $T \in \CC(B)$, which is clearly also a cut set of $\B$. Then $c_{\B}(T)=|T|+c_B(T)>|T|+1$, which contradicts the unmixedness of $J_{\B}$.

Thus, we may assume that $k <n$, i.e., in $B \setminus \{v_1,\dots,v_k\}$ is not empty. If $S \subseteq \{v_1,\dots,v_k\}$ with $1 \leq |S| \leq k$, then $c_{\B}(S) \geq |S| + 1$. Hence, Lemma \ref{L.CutsetsInUnmixedGraphs} implies $S \in \CC(\B)$ and $c_{\B}(S) = |S| + 1$, which means that $B\setminus S$ is connected. Now let $S \in \CC(B)$ containing $k_S \geq 0$ cut vertices of $\B$. Then $c_B(S) = c_{\B}(S) - k_S = |S| + 1 - k_S$, and $S \in \CC(\B)$ by Lemma \ref{L.CutsetsInUnmixedGraphs}.
\end{proof}

Next we describe our algorithm.

\begin{myalgorithm}\label{algorithm}
\ \\[1mm]
\begin{algorithm}[H]
\SetAlgoNoLine
\SetKwInOut{Input}{Input}\SetKwInOut{Output}{Output}
\Input{$n:=$ number of vertices of the blocks. \\
$k:=$ number of whiskers to attach to each block, with $4 \leq k \leq n - 3$.}
\Output{\textup{True} if every accessible block $\B$ with $n$ vertices and $k$ whiskers has a cut vertex $v$ such that $J_{\B \setminus \{v\}}$ is unmixed.
\textup{False} otherwise.}
\BlankLine
$\overbar{\mc B} := \emptyset$\\
\textit{$\mc B :=$ list of non-isomorphic blocks with $n$ vertices having neither free vertices nor vertices of degree $2$}\\
\For{$B \in \mc B$}{
    \textit{$\widetilde{\CC}(B) := \{(T, k_T): T \in \CC(B), k_T:=|T|+1-c_B(T) \}$}\\
    \If{$1 \leq k_T \leq k$ for all $(T,k_T) \in \widetilde{\CC}(B)$}{
    $V := \{v\in V(B) \ :\ |N_B(v)| \leq \lfloor \frac{n+k}{2} \rfloor -2\}$\\
    \For{$S$ subset of $V$ with cardinality $k$}{
    	\If{$|N_B(v)| \leq \lfloor \frac{n+r}{2} \rfloor - 2$ for every cut vertex $v$ of $\B$ where $r := |N_B[v] \cap S|$}{
    	$N_B(S):=\{u \in V(B)\,:\,\{u, z\} \in E(B) \text{ for some } z \in S\}$ \\
        	\If{$|N_B(S)| = n$ \rm{\textbf{and}} $B\setminus S$ is connected \rm{\textbf{and}} $(|S \cap T| = k_T$ for all $(T, k_T) \in \widetilde{\CC}(B))$}{
            \textit{$H := B[S]$ induced subgraph by $B$ on $S$}\\
           	 	\If{{\rm (}$k=4$ \rm{\textbf{and}} $H$ {\it is not a block}) \rm{\textbf{or}} ($k \neq 4$ \rm{\textbf{and}} $H$ {\it is not complete})}{
                $\B := B \cup \{\{v_i, f_i\} : v_i \in S, f_i \text{ new vertex}, i=1,\dots,k\}$\\
                	\If{$\B$ is not isomorphic to any  element in $\overbar{\mc B}$}{
                    $\overbar{\mc B} := \overbar{\mc B} \cup \{\B\}$				
                    }
                }
            }
        }
    }
}}
\For{$\B \in \overbar{\mc B}$}{
    \If{$\B$ \textit{is accessible}}{
        \If{\textit{$J_{\B \setminus \{v\}}$ is not unmixed for all $v$ cut vertex of $\B$}}{
            \Return{False}
        }
    }
}
\Return{True}
\end{algorithm}
\end{myalgorithm}

\begin{proof}
The aim of the above algorithm is to check whether for every accessible block with $n$ vertices and $k$ whiskers $\B$, there exists a cut vertex $v$ of $\B$ such that $J_{\B \setminus \{v\}}$ is unmixed. We describe the algorithm in more detail.
\begin{description}
\item[Input] We give two integers $n,k$ as input, where $n$ is the number of vertices of the blocks and $k$ the number of whiskers we add to each block, for some $4 \leq k \leq n - 3$. In fact, if $k \leq 3$ or $k \geq n - 2$, then by Propositions \ref{P.three_whiskers} and \ref{P.blocksWithTwoNonCutvertex}, respectively, there exists a cut vertex of the block with whiskers $\B$ such that $J_{\B \setminus \{v\}}$ is unmixed.
\item[Output] \textit{True} if every accessible block $\B$ with $n$ vertices and $k$ whiskers has a cut vertex $v$ such that $J_{\B \setminus \{v\}}$ is unmixed. \textit{False} otherwise.
\smallskip
\item[Line 1] We initialize $\overbar{\mc B}$ to be the empty list.

\item[Line 2] We produce the list $\mc B$ of pairwise non-isomorphic blocks with $n$ vertices that have neither free vertices nor vertices of degree $2$. We can remove blocks with free vertices by Proposition \ref{P.blockWithFreeVertex} and blocks with vertices of degree $2$ by Proposition \ref{P.blocksWithDegree2verticesAreGood}. Moreover, by Remark \ref{R.boundsNumEdgesBlocks} it is enough to consider blocks with number of edges satisfying the bounds \eqref{eq.boundsNumEdgesBlock}.
We implemented this filter as a routine in \textit{Nauty} \cite{Nauty}. 

\item[Lines 3-4] For each block $B$ in $\mc B$, we compute the set $\widetilde{\CC}(B)$ of pairs $(T, k_T)$, where $k_T$ is given by Lemma \ref{L.unmixedBlocksWithWhiskers}.

\item[Line 5] We want to analyze all possible accessible blocks with whiskers constructed from the block $B$ by adding $k$ whiskers to it. Therefore, by Lemma \ref{L.unmixedBlocksWithWhiskers}, $k_T$ will be the number of cut vertices contained in $T$. Hence, we need to have $k_T \leq k$ for every $(T, k_T) \in \widetilde{\CC}(B)$. Moreover, $k_T \geq 1$ by \cite[Lemma 4.1]{BMS22}. 

\item[Line 6] We define the set $V$ of vertices $v$ such that $|N_B(v)| \leq \lfloor \frac{n+k}{2} \rfloor - 2$ in order to use Proposition \ref{P.BoundNeighbourhood}.

\item[Line 7] We execute a for-loop over all the subsets $S$ of $V$ with cardinality $k$: we will then add a whisker to each vertex of $S$.

\item[Line 8] In light of Proposition \ref{P.BoundNeighbourhood} we check whether $|N_B(v)| \leq \lfloor \frac{n+r}{2} \rfloor - 2$ for every $v$ in $S$, where $r := |N_B[v] \cap S|$.

\item[Lines 9-10] We compute $N_B(S)$ and check the following conditions:
    \begin{itemize}
 	\item[i)] $|N_B(S)|=n$, i.e., $N_B(S)=V(B)$, otherwise there is some vertex of $B$ that is not adjacent to any vertex in $S$. Hence, the block with whiskers $\B^S$ is not accessible by \cite[Theorem 4.12]{BMS22};
    \item[ii)] $B \setminus S$ has to be connected and $|S \cap T| = k_T$ for all $(T, k_T) \in \widetilde{\CC}(B)$ by Lemma \ref{L.unmixedBlocksWithWhiskers}.
    \end{itemize}

\item[Lines 11-12] We define the induced subgraph $H := B[S]$ by $B$ on $S$. By Proposition \ref{P.FourWhiskers} we can exclude the case $k=4$ with $H$ a block and by \cite[Proposition 6.6]{BMS22} we can exclude the case $k > 4$ with $H$ a complete graph. 

\item[Line 13] We define the graph $\B:=\B^S$ obtained from $B$ by adding a whisker to each vertex of the set $S$.

\item[Lines 14-15] Finally, we add $\B$ to $\overbar{\mc B}$ if $\B$ is not isomorphic to any  element of $\overbar{\mc B}$.

\item[Lines 16-20] For each block with whiskers $\B \in \overbar{\mc B}$, if $\B$ is accessible, we consider the cut vertices of $\B$. If $J_{\B \setminus \{v\}}$ is not unmixed for every $v$ cut vertex of $\B$, we return \emph{False}. If we reach the end of the for loop, it means that for each $\B \in \overbar{\mc B}$ we found a cut vertex $v$ of $\B$ such that $J_{\B \setminus \{v\}}$ is unmixed, and hence we return \emph{True}. \qedhere
\end{description}
\end{proof}

We implemented Algorithm \ref{algorithm} combining computations in \textit{Nauty} \cite{Nauty}, routines in \emph{C}, \emph{C++} and \emph{Python}, and incorporating the code developed in \cite{LMRR21, LMRR}. We also included some optional input arguments, which allow to select the number of edges of the blocks and subdivide each computation into smaller parts. The complete code for the implementation of Algorithm \ref{algorithm} can be found in \cite{BMRS23}.

\begin{remark} \label{R.computations}
For all blocks with whiskers $\B$ with $n=|V(B)| \leq 11$, Algorithm \ref{algorithm} returns \textit{True}, i.e., for each of those graphs it finds a cut vertex $v$ such that $J_{\B \setminus \{v\}}$ is unmixed. In the computations, we only consider $4 \leq k \leq n-3$ because of Propositions \ref{P.three_whiskers} and \ref{P.blocksWithTwoNonCutvertex}; in particular $n \geq 7$.

In Table \ref{tab.filtered_accessible_graphs} we collect the number of graphs produced throughout Algorithm \ref{algorithm}. In particular, the second column contains the number of blocks in $\mc B$ (see Line 2), followed by the number of blocks with whiskers $\B$ such that $J_{\B}$ is unmixed and finally the number of accessible blocks with whiskers. Notice that none of the conditions in Theorem \ref{T.filters} applies to the graphs in Table \ref{tab.filtered_accessible_graphs} because they are all tested by Algorithm \ref{algorithm}. Part of these computations, obtained with a simpler version of the algorithm, were announced in \cite{BMRS22}.

\begin{small}
\setlength{\doublerulesep}{0pt}
\newcommand{\thickhline}{\hline\hline}
\begin{longtable}{| >{\centering\arraybackslash}m{2.1cm} " >{\centering\arraybackslash}m{1.8cm} " >{\centering\arraybackslash}m{4mm} | >{\centering\arraybackslash}m{7mm} | >{\centering\arraybackslash}m{7mm} | >{\centering\arraybackslash}m{1cm} | >{\centering\arraybackslash}m{1cm} " >{\centering\arraybackslash}m{4mm} | >{\centering\arraybackslash}m{5mm} | >{\centering\arraybackslash}m{7mm} | >{\centering\arraybackslash}m{1cm} | >{\centering\arraybackslash}m{1cm} |}
  \hline
  Number of & \multicolumn{1}{c "}{Filtered} & \multicolumn{5}{c "}{Blocks with whiskers} & \multicolumn{5}{c |}{Accessible blocks } \\
  vertices of $B$ $\downarrow$ & \multicolumn{1}{c "}{blocks} & \multicolumn{5}{c "}{with $J_{\B}$ unmixed} & \multicolumn{5}{c |}{with whiskers} \\\hline
  (Number of whiskers) $\rightarrow$ &  & 4 & 5 & 6 & 7 & 8 & 4 & 5 & 6 & 7 & 8 \\\thickhline
 7 & 79 & 0 & - & - & - & - & 0 & - & - & - & - \\\hline
 8 & 1,716 & 0 & 0 & - & - & - & 0 & 0 & - & - & - \\\hline
 9 & 61,408 & 0 & 2 & 2 & - & - & 0 & 2 & 1 & - & - \\ \hline
10 & 4,054,291 & 0 & 6 & 9 & 25 & - & 0 & 2 & 5 & 24 & - \\\hline
11 & 475,625,326 & 0 & 54 & 541 & 367 & 53 & 0 & 13 & 309 & 283 & 38 \\\hline
\caption{Filtered accessible blocks with whiskers} \label{tab.filtered_accessible_graphs}
\end{longtable}
\end{small}

We also remark that the graphs in Table \ref{tab.filtered_accessible_graphs} are a small subset of all accessible graphs. For instance, in \cite[Table 1]{LMRR21} the authors found $24,270$ connected accessible graphs with up to $12$ vertices, whereas we found only $5$ accessible graphs with up to $15$ vertices.

Notice that the blocks with whiskers with $J_{\B}$ unmixed in Table \ref{tab.filtered_accessible_graphs} are exactly the graphs in $\overbar{\mathcal B}$ (see Line 15). It is important that the number of graphs in $\overbar{\mathcal B}$ is small because testing unmixedness is the bottleneck of our algorithm.
\end{remark}

\begin{example}
The graph $\B$ in Figure \ref{F.graph_all_filters} is one of the smallest accessible blocks with whiskers from Table \ref{tab.filtered_accessible_graphs}. This also means that the graph $\B$ does not satisfy any of the conditions in Theorem \ref{T.filters}. However, in this case $J_{\B \setminus \{1\}}$ is unmixed.

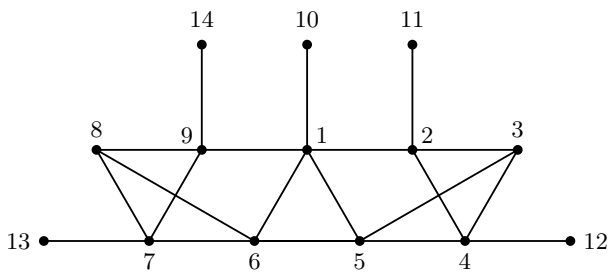
\begin{figure}[ht!]
\begin{tikzpicture}[scale=0.7]
\node[label={above right:{\small $1$}}] (a) at (1,1.7320508075688776) {};
\node[label={above right:{\small $2$}}] (b) at (3,1.7320508075688779) {};
\node[label={above:{\small $3$}}] (c) at (5,1.7320508075688779) {};
\node[label={below:{\small $4$}}] (d) at (4,0) {};
\node[label={below:{\small $5$}}] (e) at (2,0) {};
\node[label={below:{\small $6$}}] (f) at (0,0) {};
\node[label={below:{\small $7$}}] (g) at (-2,0) {};
\node[label={above:{\small $8$}}] (h) at (-3,1.7320508075688779) {};
\node[label={above left:{\small $9$}}] (i) at (-1,1.7320508075688779) {};
\node[label={above:{\small $10$}}] (j) at (1,3.7320508075688785) {};
\node[label={above:{\small $11$}}] (k) at (3,3.732050807568878) {};
\node[label={right:{\small $12$}}] (l) at (6,0) {};
\node[label={left:{\small $13$}}] (m) at (-4,0) {};
\node[label={above:{\small $14$}}] (n) at (-1,3.732050807568878) {};
\draw (0,0)-- (2,0);
\draw (2,0)-- (1,1.7320508075688776);
\draw (1,1.7320508075688776)-- (0,0);
\draw (-4,0)-- (6,0);
\draw (-3,1.7320508075688779)-- (5,1.7320508075688779);
\draw (3,3.732050807568878)-- (3,1.7320508075688779);
\draw (1,3.7320508075688785)-- (1,1.7320508075688776);
\draw (-1,3.732050807568878)-- (-1,1.7320508075688779);
\draw (-3,1.7320508075688779)-- (-2,0);
\draw (-1,1.7320508075688779)-- (-2,0);
\draw (-3,1.7320508075688779)-- (0,0);
\draw (5,1.7320508075688779)-- (4,0);
\draw (2,0)-- (5,1.7320508075688779);
\draw (3,1.7320508075688779)-- (4,0);
\end{tikzpicture}
\caption{An accessible block with whiskers from Table \ref{tab.filtered_accessible_graphs}} \label{F.graph_all_filters}
\end{figure}
\end{example}

\begin{example}
Looking at Table \ref{tab.filtered_accessible_graphs}, one can see that Algorithm \ref{algorithm} did not find any block with $4$ whiskers and $n \leq 11$ in $\overbar{\mathcal B}$ (see Line 15). However, it found non-accessible blocks with $12$ vertices and $4$ whiskers whose binomial edge ideal is unmixed, for instance the graph in Figure \ref{F.unmixed_four_whiskers}. We do not know if there are accessible blocks with $4$ whiskers in $\overbar{\mathcal B}$ for $n \geq 12$.

\begin{figure}[ht!]
\begin{tikzpicture}[scale=0.75]
\node[label={above left:{\small $1$}}] (a) at (-2,2) {};
\node[label={above left:{\small $2$}}] (a) at (0,2) {};
\node[label={above right:{\small $3$}}] (a) at (2,2) {};
\node[label={above right:{\small $4$}}] (a) at (4,2) {};
\node[label={right:{\small $5$}}] (a) at (5.7210783496401945,0.98123147162872) {};
\node[label={right:{\small $6$}}] (a) at (5.420710894868759,-0.604201345485154) {};
\node[label={below:{\small $7$}}] (a) at (3.3960403380615816,0.6005510943625672) {};
\node[label={below:{\small $8$}}] (a) at (3.7187382017306003,-1.02271158882251) {};
\node[label={below:{\small $9$}}] (a) at (2,0) {};
\node[label={below:{\small $10$}}] (a) at (0,0) {};
\node[label={below:{\small $11$}}] (a) at (-1.7214994599833984,-1.0180567809689534) {};
\node[label={left:{\small $12$}}] (a) at (-3.7153255330849273,0.9715748371675685) {};

\node[label={above:{\small $13$}}] (a) at (-2,4) {};
\node[label={above:{\small $14$}}] (a) at (0,4) {};
\node[label={above:{\small $15$}}] (a) at (2,4) {};
\node[label={above:{\small $16$}}] (a) at (4,4) {};

\draw (0,0)-- (2,0);
\draw (2,0)-- (2,2);
\draw (2,2)-- (0,2);
\draw (0,2)-- (0,0);
\draw (0,4)-- (0,2);
\draw (2,4)-- (2,2);
\draw (-2,4)-- (-2,2);
\draw (-2,2)-- (0,2);
\draw (2,2)-- (4,2);
\draw (4,4)-- (4,2);
\draw (-2,2)-- (-3.7153255330849273,0.9715748371675685);
\draw (0,0)-- (-1.7214994599833984,-1.0180567809689534);
\draw (-3.7153255330849273,0.9715748371675685)-- (-1.7214994599833984,-1.0180567809689534);
\draw (-2,2)-- (0,0);
\draw (-1.7214994599833984,-1.0180567809689534)-- (0,2);
\draw (-3.7153255330849273,0.9715748371675685)-- (0,2);
\draw (4,2)-- (5.7210783496401945,0.98123147162872);
\draw (2,0)-- (3.7187382017306003,-1.02271158882251);
\draw (3.3960403380615816,0.6005510943625672)-- (4,2);
\draw (3.3960403380615816,0.6005510943625672)-- (2,2);
\draw (3.3960403380615816,0.6005510943625672)-- (5.7210783496401945,0.98123147162872);
\draw (2,2)-- (3.7187382017306003,-1.02271158882251);
\draw (3.7187382017306003,-1.02271158882251)-- (5.7210783496401945,0.98123147162872);
\draw (4,2)-- (3.7187382017306003,-1.02271158882251);
\draw (3.7187382017306003,-1.02271158882251)-- (5.420710894868759,-0.604201345485154);
\draw (5.420710894868759,-0.604201345485154)-- (4,2);
\end{tikzpicture}
\caption{A block with four whiskers $\B$ with $J_{\B}$ unmixed} \label{F.unmixed_four_whiskers}
\end{figure}
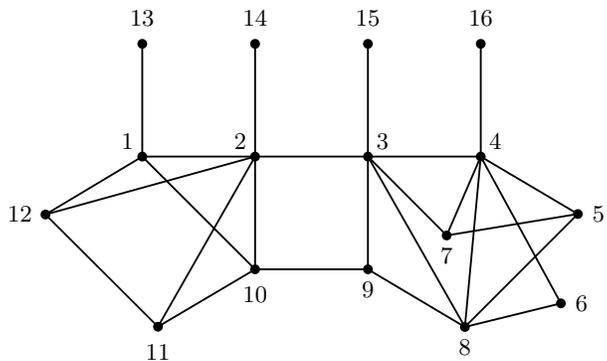

\end{example}

In \cite{LMRR} the authors verified that Conjecture \ref{conjecture} holds for all graphs with up to $12$ vertices. Algorithm \ref{algorithm} allows us to go way beyond, checking the conjecture for all graphs with up to $15$ vertices and for blocks with up to $11$ vertices with whiskers (reaching some graphs with $11+8=19$ vertices), as we prove next.

\begin{theorem} \label{T.computational_results}
Let $G$ be one of the following:
\begin{itemize}
\item[(a)] a graph with $|V(G)| \leq 15$;
\item[(b)] $G = \B$ a block with whiskers, where $n=|V(B)| \leq 11$.
\end{itemize}
Then the conditions below are equivalent:
\begin{enumerate}
\item $J_G$ is strongly unmixed;
\item $R/J_G$ is Cohen-Macaulay;
\item $R/J_G$ satisfies Serre's condition $(S_2)$;
\item $G$ is accessible.		
\end{enumerate}
\end{theorem}

\begin{proof}
We know that $(1) \Rightarrow (2)\Rightarrow (3)\Rightarrow (4)$ for every graph $G$, hence we only need to show that $(4)\Rightarrow (1)$.

\textit{(a)} Let $G$ be a graph with $|V(G)| \leq 15$. By Remark \ref{R.ReductionFewVertices}, it is enough to show that every block with whiskers $\B$ with at most $15$ vertices has a cut vertex $v$ such that $J_{\B \setminus \{v\}}$ is unmixed. If $\B$ has at most $3$ whiskers, then the claim follows by Proposition \ref{P.three_whiskers}. Otherwise, $B$ is a block with at most $11$ vertices and we conclude by Remark \ref{R.computations}.

\textit{(b)} Let $G = \B$ be a block with whiskers, where $n = |V(B)| \leq 11$. We proceed by induction on $|V(B)| \geq 1$. The case $|V(B)|=1$ is trivial. By Remark \ref{R.computations}, there exists a cut vertex, say $v_1$, such that $J_{\B \setminus \{v_1\}}$ is unmixed. Moreover, in light of Proposition \ref{P.stronglyUnmixed}, it is enough to show that $J_{\B \setminus \{v_1\}}$ and $J_{\B_{v_1} \setminus \{v_1\}}$ are strongly unmixed. Notice that the graphs $\B \setminus \{v_1\}$ and $\B_{v_1} \setminus \{v_1\}$ are accessible by \cite[Corollary 5.16]{BMS22}. Furthermore, for each block $C$ of $\B \setminus \{v_1\}$, $J_{\overbar C}$ is strongly unmixed by induction since $|V(C)| < |V(B)|$, and hence $J_{\B \setminus \{v_1\}}$ is strongly unmixed by \cite[Theorem 3.17]{SS22}.

Let now $G_1=\B_{v_1} \setminus \{v_1\}$, which is an accessible block $C$ with $k-1$ whiskers. We only need to show that $J_{G_1}$ is strongly unmixed. Notice that $V(C)=(V(B) \cup \{f_1\}) \setminus \{v_1\}$, where $f_1$ is the leaf adjacent to $v_1$ in $\B$. Therefore, $|C|=|B| \leq 11$ and by assumption there exists a cut vertex, say $v_2$ of $G_1$ such that $J_{G_1 \setminus \{v_2\}}$ is unmixed. As above, $J_{G_1 \setminus \{v_2\}}$ is strongly unmixed and $G_2=(G_1)_{v_2} \setminus \{v_2\}$ is an accessible block with $k-2$ whiskers. Hence, it is enough to prove that $J_{G_2}$ is strongly unmixed. Repeating the same argument we obtain the graph $G_k=(G_{k-1})_{v_k} \setminus \{v_k\}$ and we only need to show that $J_{G_k}$ is strongly unmixed. However, $G_{k-1}$ is an accessible block with one whisker, and then its unique cut vertex is adjacent to all the other vertices of $G_{k-1}$ by \cite[Lemma 4.9]{BMS22}. Therefore, $G_k$ is a complete graph and $J_{G_k}$ is strongly unmixed.
\end{proof}

{\bf Acknowledgements}. Part of this paper was written while the first and the third authors visited the Discrete Mathematics and Topological Combinatorics research group at Freie Universit\"{a}t Berlin. They want to express their thanks for the hospitality.

\end{document}